\documentclass{amsart}

\usepackage{xcolor}
\usepackage{setspace}
\usepackage{enumerate}

\newtheorem{theorem}{Theorem}[section]
\newtheorem{prop}[theorem]{Proposition}

\newtheorem{rem}[theorem]{Remark}

\DeclareMathOperator{\osc}{{\rm osc}}

\numberwithin{equation}{section}
\title[$p$-variation of Riemann's function]{On the $p$-variation of Riemann's ``nondifferentiable'' function}
\author{Martin Lind}
\address{Department of Mathematics and Computer Science, Karlstad University, Universitetsgatan 2, 65188 Karlstad, Sweden}
\email{martin.lind@kau.se}



\begin{document}

\begin{abstract}
    We investigate the variational properties of Riemann's ``nondifferentiable'' function $R$. We show that $R$ has finite $p$-variation for $p>4/3$ and infinite $p$-variation for  $p<4/3$. As an application of our results, we provide a new perspective on an upper bound of the Hausdorff dimension of the image of $R$, due to Eceizabarrena.
\end{abstract}
\maketitle

\section{Introduction}
Riemann proposed the function
\begin{equation}
    \label{riemann}
    R(x)=\sum_{n=1}^\infty\frac{\sin(\pi n^2x)}{n^2}
\end{equation}
as a potential example of an everywhere continuous but nowhere differentiable function. Hardy \cite{Ha1916} proved that $R$ is non-differentiable a.e., while Gerver \cite{Ge70} showed that $R$ is differentiable at certain rationals, thereby disproving Riemann's conjecture that $R$ is nowhere differentiable. The remarkably intricate regularity properties of Riemann's function were uncovered by Jaffard \cite{Ja96}, building on earlier work of Duistermaat \cite{Du91}.

Variational properties of Riemann's function has not been investigated before. In this paper, we shall focus on the \emph{$p$-variation} of $R$ (in the sense of Wiener).

Let $f:\mathbb{R}\rightarrow\mathbb{C}$ be periodic with period $P$. For any $E\subset\mathbb{R}$ denote
\begin{equation}
    \nonumber
    \osc(f;E)=\sup_{x,y\in E}|f(x)-f(y)|.
\end{equation}
For $1\le p<\infty$, the \emph{$p$-variation of $f$} is defined by
\begin{equation}
    \label{p-var}
    \mathcal{V}_p(f;[0,P])=\sup_{\mathcal{I}}\left(\sum_{I\in\mathcal{I}}\osc(f;I)^p\right)^{1/p},
\end{equation}
where the supremum is taken over all finite families $\mathcal{I}=\{I\}$ of disjoint intervals contained in $[0,P]$. Denote by $\mathcal{BV}_p(0,P)$ the class of all functions for which (\ref{p-var}) is finite. For $p=1$, (\ref{p-var}) is the classical total variation of Jordan while for $p>1$, the notion of $p$-variation originates in the works of Wiener \cite{Wi24} and L.C. Young \cite{Yo36}. Note that for $1\le p<q<\infty$, the following embedding holds:
\begin{equation}
    \label{increasing}
    \mathcal{BV}_p(0,P)\subset\mathcal{BV}_q(0,P)
\end{equation}
For technical reasons, it is more convenient to work with the function
\begin{equation}
    \label{simpleRiemann}
    \Phi(x)=\sum_{n=1}^\infty\frac{e^{\pi in^2x}}{\pi i n^2}.
\end{equation}
Note that $\Phi$ and $R$ are 2-periodic, and that
\begin{equation}
    \label{phi-r-relation}
    \pi\Phi(x)=R(x)+i\widetilde{R}(x),
\end{equation}
where $\widetilde{R}$ denotes the conjugate Fourier series of (\ref{riemann}) (see (\ref{conjugate}) below). We first prove the following result, relating the variation of $R$ to the variation of $\Phi$.
\begin{theorem}
    \label{phi-r-relationTeo}
    For $p\ge1$ there holds
    \begin{equation}
        \label{equiv}
        \mathcal{V}_p(\Phi;[0,2])<\infty\quad\Leftrightarrow\quad\mathcal{V}_p(R;[0,2])<\infty.
    \end{equation}
\end{theorem}
The fact that $\Phi\in\mathcal{BV}_p(0,2)$ implies that $R\in\mathcal{BV}_p(0,2)$ is more or less clear. The converse is more involved, since it is not obvious that $R\in\mathcal{BV}_p(0,2)$ implies that $\widetilde{R}\in\mathcal{BV}_p(0,2)$. Nevertheless, the last statement is true and we prove it below, using a curious representation formula for $\widetilde{R}$ in terms of $R$.

It is not difficult to show that for $x,y\in\mathbb{R}$, there holds
\begin{equation}
    \label{1/2-holder}
    |\Phi(x)-\Phi(y)|\le 4\sqrt{|x-y|},
\end{equation}
see \cite[p. 316]{Du91}. Clearly (\ref{1/2-holder}) implies that $\osc(\Phi;I)\le 4\sqrt{|I|}$, whence $\Phi\in\mathcal{BV}_2(0,2)$. At the same time, $\Phi$ is non-differentiable a.e. and consequently $\Phi\notin\mathcal{BV}_1(0,2)$. A natural problem is to determine precisely for which values of $p>1$ the function $\Phi$ belongs to $\mathcal{BV}_p(0,2)$. 
\begin{theorem}
    \label{simpleResult}
    The following statements are true.
    \begin{enumerate}[(a)]
        \item\label{simpleResult2} If $p<4/3$, then $\Phi\notin\mathcal{BV}_p(0,2)$.
        \item\label{simpleResult1} If $p>4/3$, then $\Phi\in\mathcal{BV}_p(0,2)$.
    \end{enumerate}
\end{theorem}
Eceizabarrena \cite{Ec21} proved estimates for the Hausdorff dimension ${\rm dim}_H$ of the image of $\Phi([0,2])\subset\mathbb{C}$. In particular, the upper bound
\begin{equation}
    \label{EciIneq}
    {\rm dim}_H\Phi([0,2])\le\frac{4}{3}
\end{equation}
was obtained. Our result provides another perspective on (\ref{EciIneq}), see Remark \ref{Eceizabarrena} below. 

Regarding the proof of Theorem \ref{simpleResult}, statement (\ref{simpleResult2}) follows from an observation of Hardy \cite{Ha1916} that for a.e. $x$,
\begin{equation}
    \nonumber
    \Phi(x+h)-\Phi(x)\neq o(h^{3/4}),
\end{equation}
together with Proposition \ref{exceptional} below.
Statement (\ref{simpleResult1}) follows from an embedding theorem of Terekhin \cite{Te67} together with an estimate of the $L^p$ modulus of continuity of $\Phi$ (Proposition \ref{modulusLp} below).

Note that Theorem \ref{simpleResult} leaves the critical value $p=4/3$ open.
We conjecture that $\Phi\in\mathcal{BV}_{4/3}(0,2)$; at the moment we are not able to prove this. 

\section{Variation of of $R$ and $\Phi$: Proof of Theorem \ref{phi-r-relationTeo}}

Let $h\in\mathbb{R}$ and $\lambda>0$. Translation and dilation operators are defined as usual:
\begin{equation}
    \nonumber
    \tau_hf(x)=f(x+h),\quad \delta_\lambda f(x)=f(\lambda x).
\end{equation}
It is not difficult to verify that
\begin{equation}
    \label{invar1}
    \mathcal{V}_p(\tau_hf;[0,P])=\mathcal{V}_p(f;[0,P])
\end{equation}
and for $m\in\mathbb{N}$, then
\begin{equation}
    \label{invar2}
    \mathcal{V}_p(\delta_m f;[0,P])\le m^{1/p}\mathcal{V}_p(f;[0,P]).
\end{equation}

\begin{proof}[Proof of Theorem \ref{phi-r-relationTeo}]
    The conjugate Fourier series of (\ref{riemann}) is
    \begin{equation}
        \label{conjugate}
        \widetilde{R}(x)=-\sum_{n=1}^\infty\frac{\cos(\pi n^2 x)}{n^2}.
    \end{equation}
    By (\ref{phi-r-relation}), it is clear that $\Phi\in\mathcal{BV}_p(0,2)$, then $R$ and $\widetilde{R}$ belong to $\mathcal{BV}_p(0,2)$; this proves the ''$\Rightarrow$''-direction of (\ref{equiv}). Similarly, if $R,\widetilde{R}\in\mathcal{BV}_p(0,2)$, then $\Phi\in\mathcal{BV}_p(0,2)$. For the converse implication, however, we only assume that $R\in\mathcal{BV}_p(0,2)$; thus we must demonstrate that $R\in\mathcal{BV}_p(0,2)$ implies $\widetilde{R}\in\mathcal{BV}_p(0,2)$. Assume that $R\in\mathcal{BV}_p(0,2)$.
    Note that for any $n\in\mathbb{N}$ 
    \begin{equation}
        \nonumber
        \sin(\pi n^2x+\pi n^2/2)-\sin(\pi n^2x-\pi n^2/2)=2\cos(\pi n^2x)\sin(\pi n^2/2).
    \end{equation}
    If $n$ is even, then $\sin(n^2\pi/2)=0$. If $n$ is odd, then  $n^2\equiv 1\pmod{4}$ and therefore $\sin(\pi n^2/2)=1$. Hence,
    \begin{equation}
        \label{identityPhi}
        R(x+1/2)-R(x-1/2)=2\sum_{m\text{ odd}}\frac{\cos(\pi m^2x)}{m^2}.
    \end{equation}
    On the other hand, each $n\in\mathbb{N}$ can be written uniquely as $n=2^lm$ for some $l\ge0$ and odd $m$. Thus,
    \begin{eqnarray}
        \nonumber
        \widetilde{R}(x)&=&-\sum_{l=0}^\infty\sum_{m\text{ odd}}\frac{\cos(\pi 2^{2l}m^2x)}{2^{2l}m^2}=-\sum_{l=0}^\infty 2^{-2l}\sum_{m\text{ odd}}\frac{\cos(\pi 2^{2l}m^2x)}{m^2}\\
        \nonumber
        &=&-\sum_{l=0}^\infty 2^{-2l-1}\left(R(2^{2l}x+1/2)-R(2^{2l}x-1/2)\right),
    \end{eqnarray}
    where we used (\ref{identityPhi}) in the second line. Denote
    \begin{equation}
        \nonumber
        \rho_l(x)=R(2^{2l}x+1/2)-R(2^{2l}x-1/2).
    \end{equation}
    By the assumption $R\in\mathcal{BV}_p(0,2)$, we use (\ref{invar1}) and (\ref{invar2}) to obtain
    \begin{equation}
        \nonumber
        \mathcal{V}_p(\rho_l;[0,2])\le 2^{2l/p+1}\mathcal{V}_p(R;[0,2]).
    \end{equation}
    Since $\mathcal{V}_p(\cdot;[0,2])$ is a homogeneous seminorm, we have 
    \begin{eqnarray}
        \nonumber
        \mathcal{V}_p(\widetilde{R};[0,2])&\le&\sum_{l=0}^\infty2^{-2l-1}\mathcal{V}_p(\rho_l;[0,2])\le\sum_{l=0}^\infty 2^{-2l(1-1/p)}\mathcal{V}_p(R;[0,2])\\
        \nonumber
        &=&c_p\mathcal{V}_p(R;[0,2]),
    \end{eqnarray}
    which is finite by the assumption $R\in\mathcal{BV}_p(0,2)$.
\end{proof}

\section{A criterion for infinite variation; proof of Theorem \ref{simpleResult}(\ref{simpleResult2})}
To prove Theorem \ref{simpleResult}(\ref{simpleResult2}), we shall make use of the following result, which may be of independent interest.
\begin{prop}
    \label{exceptional}
    Let $f:\mathbb{R}\rightarrow\mathbb{C}$ be $P$-periodic and let $\beta\in(0,1]$. Set
    \begin{equation}
        \nonumber
        E(f,\beta)=\left\{x\in[0,P]:\limsup_{h\rightarrow0}\frac{|f(x+h)-f(x)|}{|h|^\beta}=\infty\right\}.
    \end{equation}
    If $E(f,\beta)$ has positive outer Lebesgue measure, then $f\notin\mathcal{BV}_{1/\beta}(0,P)$.
\end{prop}
\begin{proof}
    Assume that $E(f,\beta)$ has positive outer Lebesgue measure. Take $\epsilon>0$ such that 
    \begin{equation}
        \nonumber
        |E(f,\beta)\cap (\epsilon, P-\epsilon)|_e>0,
    \end{equation}
    where $|\cdot|_e$ denotes outer Lebesgue measure.
    Let $M>0$ be arbitrary. For each $x\in E(f,\beta)\cap(\epsilon,P-\epsilon)$ there is $h_x$ with
    \begin{equation}
        \nonumber
        0<|h_x|<\epsilon
    \end{equation}
    such that
    \begin{equation}
        \nonumber
        |f(x+h_x)-f(x)|>M|h_x|^\beta.
    \end{equation}
    Denote 
    \begin{equation}
        \nonumber
        I_x=[x-|h_x|,x+|h_x|],
    \end{equation}
    then $\{I_x:x\in E(f,\beta)\cap (\epsilon,P-\epsilon)\}$ covers $E(f,\beta)\cap (\epsilon,P-\epsilon)$. By Vitali's simple covering lemma \cite[p. 102]{WeZy77}, there is a number $c_0>0$ and a finite subcollection $\{I_{x_j}:1\le j\le N\}$ of disjoint intervals such that 
    \begin{equation}
        \nonumber
        c_0|E(f,\beta)\cap(\epsilon,P-\epsilon)|_e\le\sum_{j=1}^N|I_{x_j}|.
    \end{equation}
    On the other hand, 
    \begin{equation}
        \nonumber
        |I_{x_j}|=2|h_{x_j}|<2M^{-1/\beta}|f(x_j+h_{x_j})-f(x_j)|^{1/\beta}.
    \end{equation}
    Whence,
    \begin{eqnarray}
        \nonumber
        c_0|E(f,\beta)\cap(\epsilon,P-\epsilon)|_e&\le&2\sum_{j=1}^N|h_{x_j}|\le 2M^{-1/\beta}\sum_{j=1}^N|f(x_j+h_{x_j})-f(x_j)|^{1/\beta}\\
        \nonumber
        &\le& 2M^{-1/\beta}\mathcal{V}_{1/\beta}(f;[0,P])^{1/\beta} .
    \end{eqnarray}
    Consequently,
    \begin{equation}
        \nonumber
        \mathcal{V}_{1/\beta}(f;[0,P])\ge M\left(\frac{c_0|E(f,\beta)\cap(\epsilon,P-\epsilon)|_e}{2}\right)^\beta.
    \end{equation}
    Since $M$ was arbitrary, we obtain $f\notin\mathcal{BV}_{1/\beta}(0,P)$.
\end{proof}\begin{proof}[Proof of Theorem \ref{simpleResult}(\ref{simpleResult2})]
    Hardy \cite[p. 323]{Ha1916} observed that 
    \begin{equation}
        \nonumber
        \liminf_{h\rightarrow0}\frac{|\Phi(x+h)-\Phi(x)|}{|h|^{3/4}}>0
    \end{equation}
    for every irrational $x$. Take any $\beta>3/4$. Then
    \begin{equation}
        \nonumber
        \limsup_{h\rightarrow0}\frac{|\Phi(x+h)-\Phi(x)|}{|h|^\beta}=\infty
    \end{equation}
    for a.e. $x$. In other words, $E(\Phi;\beta)$ has full measure in $[0,2]$ for any $\beta>3/4$. By Proposition \ref{exceptional}, $\Phi\notin\mathcal{BV}_{1/\beta}(0,2)$ for any $\beta>3/4$, or equivalently, $\Phi\notin\mathcal{BV}_p(0,2)$ for any $p<4/3$.
\end{proof}

\section{Integral smoothness of $\Phi$; proof of Theorem \ref{simpleResult}(\ref{simpleResult1})}

To prove Theorem \ref{simpleResult}(\ref{simpleResult1}), we make use of the integral smoothness properties of $\Phi$. 
Denote by $L^p(0,P)$ the set of all $P$-periodic functions $f:\mathbb{R}\rightarrow\mathbb{C}$ such that
\begin{equation}
    \nonumber
    \|f\|_{L^p(0,P)}=\left(\int_0^P|f(x)|^p\mathrm{d}x\right)^{1/p}\quad (1\le p<\infty).
\end{equation}
The \emph{$L^p$-modulus of continuity} of $f\in L^p(0,P)$ is defined by
\begin{equation}
    \nonumber
    \omega(f;\delta)_p=\sup_{0<h\le\delta}\|\Delta_hf\|_{L^p(0,P)},
\end{equation}
where $\Delta_hf(x)=(\tau_h-I)f(x)=f(x+h)-f(x)$. Terekhin \cite{Te67} (see also \cite{KoLi09}) showed that for any $p>1$
\begin{equation}
    \label{modulus1}
    \mathcal{V}_p(f;[0,P])\le C\int_0^Pt^{-1/p}\omega(f;t)_p\frac{\mathrm{d}t}{t}.
\end{equation}
We remark that the right-hand side of (\ref{modulus1}) is the standard seminorm in the Besov space $B_{p,1}^{1/p}(0,P)$.
\begin{prop}
    \label{modulusLp}
    For $1\le p\le2$ and any $\delta\in(0,1]$, there holds
    \begin{equation}
        \label{modulusIneq}
        \omega(\Phi;\delta)_p\le 6\delta^{3/4}.
    \end{equation}
\end{prop}
\begin{proof}
We treat first the case $p=2$. Let $h\in(0,1]$ be arbitrary. For any $x\in\mathbb{R}$, there holds
\begin{equation}
    \nonumber
    \Phi(x+h)-\Phi(x)=\sum_{n=1}^\infty(e^{\pi i n^2h}-1)\frac{e^{\pi i n^2 x}}{\pi in^2}.
\end{equation}
Applying Parseval's identity
\begin{equation}
    \label{l2norm}
    \|\Delta_h\Phi\|^2_{L^2(0,2)}=2\sum_{n=1}^\infty\frac{|e^{\pi i n^2h}-1|^2}{\pi^2 n^4}.
\end{equation}
Note that $|e^{iz}-1|\le\min(2,|z|)$, whence
\begin{equation}
    \label{elemIneq}
    |e^{\pi i n^2h}-1|\le\min(2,\pi n^2h).
\end{equation}
Let $N=\lfloor h^{-1/2}\rfloor$, use (\ref{elemIneq}), and split the sum at the right-hand side of (\ref{l2norm}) as follows:
\begin{eqnarray}
    \nonumber
    \sum_{n=1}^\infty\frac{|e^{\pi i n^2h}-1|^2}{\pi^2 n^4}&\le&\sum_{n=1}^\infty\frac{\min(2,\pi n^2h)^2}{\pi^2n^4}\\
    \nonumber
    &\le &\sum_{n\le N}\frac{\pi^2 n^4h^2}{\pi^2 n^4}+\sum_{n> N}\frac{4}{\pi^2 n^4}\\
    \nonumber
    &=&h^2\sum_{n\le N}1+\frac{4}{\pi^2}\sum_{n> N}\frac{1}{n^4}\\
    \nonumber
    &=& h^2N+\frac{4}{3\pi^2N^3}.
\end{eqnarray}
Since $N\le h^{-1/2}$, we get $h^2N\le h^{3/2}$. Furthermore, since $h^{-1/2}\ge1$ and $N=\lfloor h^{-1/2}\rfloor$, we have $N\ge h^{-1/2}/2$. It follows that $1/N^3\le 8h^{3/2}$ and consequently $4/(3\pi^2 N^3)\le (32/3\pi^2)h^{3/2}$. Combining the preceding estimates yields
\begin{equation}
    \nonumber
    \|\Delta_h\Phi\|^2_{L^2(0,2)}\le 2\left(1+\frac{32}{3\pi^2}\right)h^{3/2}.
\end{equation}
Hence, $\omega(\Phi;\delta)_2\le3\delta^{3/4}$ for any $\delta\in(0,1]$. Furthermore, for $1\le p\le 2$
\begin{equation}
    \nonumber
    \|f\|_{L^p(0,2)}\le 2^{1/p-1/2}\|f\|_{L^2(0,2)}.
\end{equation}
Hence, for $1\le p\le 2$ and $\delta\in(0,1]$
\begin{equation}
    \nonumber
    \omega(\Phi;\delta)_p\le5\delta^{3/4},
\end{equation}
thus proving (\ref{modulusIneq}).
\end{proof}
\begin{proof}[Proof of Theorem \ref{simpleResult}(\ref{simpleResult1})]
    Let $4/3<p\le 2$. By (\ref{modulus1}) and Proposition \ref{modulusLp}, 
    \begin{equation}
        \nonumber
        \mathcal{V}_p(\Phi;[0,2])\le C\int_0^1t^{3/4-1/p-1}\mathrm{d}t+C\int_1^2t^{-1/p}\omega(\Phi;t)_p\frac{\mathrm{d}t}{t}.
    \end{equation}
    The first integral is finite, since $4/3-1/p>0$, while the second is finite since
    \begin{equation}
        \nonumber
        \omega(\Phi;t)_p\le 2\|\Phi\|_{L^p(0,2)}.
    \end{equation}
    Hence, $\Phi\in\mathcal{BV}_p(0,2)$ for $4/3<p\le2$. For $p>2$, the conclusion follows from $\Phi\in\mathcal{BV}_2(0,2)$ and the embedding (\ref{increasing}).
\end{proof}

\begin{rem}
    \label{Eceizabarrena}
    Assume that $f:\mathbb{R}\rightarrow\mathbb{C}$ is $P$-periodic. It appears to be well-known (see, e.g., \cite[Eq. (9)]{Ya18}) that if  $f\in\mathcal{BV}_p(0,P)$, then
    \begin{equation}
    \label{dimEst}
    {\rm dim}_Hf([0,P])\le\min(p,2).
    \end{equation}
    (We identify $\mathbb{C}$ with $\mathbb{R}^2$.)
    
    By Theorem \ref{simpleResult}(\ref{simpleResult1}), $\Phi\in\mathcal{BV}_{4/3+\epsilon}(0,2)$ for any $\epsilon>0$. It follows from (\ref{dimEst}) that
    \begin{equation}
        \nonumber
        \dim_H\Phi([0,2])\le \frac{4}{3}+\epsilon
    \end{equation}
    for any $\epsilon>0$. Letting of $\epsilon\rightarrow0+$ gives (\ref{EciIneq}).
\end{rem}

\bibliographystyle{plain}
\bibliography{ref3}

@article{Du91,
    author = {Duistermaat, J. J.},
    title = {{Selfsimilarity of 'Riemann's nondifferentiable function'}},
    journal = {Nieuw Arch. Wisk.},
    year = 1991,
    volume={9},
    number={4},
    pages={303--337}
}

@article{Ha1916,
    author ={Hardy, G. H.} ,
    title = {{Weierstrass's Non-Differentiable Function}},
    journal ={Trans. Amer. Math. Soc.} ,
    year = {1916},
    volume={17},
    number={3},
    pages={301--325}
}

@article{Ja96,
    author ={Jaffard, S.} ,
    title = {{The spectrum of singularities of Riemann's function}},
    journal ={Rev. Math. Iberoamericana} ,
    year = {1996},
    volume={12},
    number={2},
    pages={441--460}
}

@article{Ec21,
    author ={Eceizabarrena, D.} ,
    title = {{On the Hausdorff dimension of Riemann’s non-differentiable function}},
    journal ={Trans. Amer. Math. Soc.} ,
    year = {2021},
    volume={374},
    pages={7679--7713}
}

@article{Ya18,
    author ={Yao, X.} ,
    title = {{Hausdorff Dimension of the Range and the Graph of Stable-Like Processes}},
    journal ={J. Theor. Probab.} ,
    year = {2018},
    volume={31},
    pages={2412-–2431}
}

@article{Te67,
  author    = {A. P. Terekhin},
  title     = {{Integral smoothness properties of periodic functions of bounded $p$-variation}},
  journal ={Math. Notes},
  year = {1967},
    volume={2},
    pages={659--665}
}

@article{Ge70,
  author    = {Gerver, J.},
  title     = {{The Differentiability of the Riemann Function at Certain Rational Multiples of $\pi$}},
  journal ={Amer. J. Math.},
  year = {1970},
    volume={92},
    number={1},
    pages={33--55}
}

@article{KoLi09,
    author = {Kolyada, V. I. and Lind, M.},
    title = {On functions of bounded $p$-variation},
    journal ={J. Math. Anal. Appl.},
    volume={365},
    number={2},
    pages={582--604},
    year ={2009}
}

@article{Wi24,
    author = {Wiener, N.},
    title = {{The quadratic variation of a function and its Fourier coefficients}},
    journal ={Massachusetts
    J. Math. and Phys.},
    volume={3},
    pages={72--94},
    year ={1924}
}

@article{Yo36,
    author = {Young, L. C.},
    title = {{An inequality of the Hölder type, connected with Stieltjes integration}},
    journal ={Acta Math.},
    volume={67},
    pages={251--282},
    year ={1936}
}

@book{WeZy77,
  title     = {Measure and Integral},
  author    = {Weeden, R. L. and Zygmund, A.},
  year      = {1977},
  publisher = {Marcel Dekker, Inc.}
}

\end{document}